\documentclass{amsart}
\usepackage{amsmath,amssymb,amscd,latexsym,amsxtra,amsfonts,bbold,amsthm}
\usepackage{tkz-euclide}
\usetkzobj{all}
\usepackage{nicefrac}
\usepackage[scr=boondoxo,scrscaled=1.05]{mathalfa}
\usepackage{enumerate}
\usepackage{leftidx}
\usepackage{dsfont}
\usepackage[all]{xy}
\usepackage{graphicx}
\usepackage{relsize}

\usepackage[mathscr]{eucal}
\usepackage{calrsfs}
\usepackage{fge}
\usepackage{pst-all}
\usepackage{mathtools}
\usepackage[top=2cm, bottom=2cm, left=1.5cm, right=1.5cm]{geometry}

\usepackage[ps2pdf]{hyperref}
\definecolor{britishracinggreen}{rgb}{0.0, 0.26, 0.15}
\definecolor{amaranth}{rgb}{0.9, 0.17, 0.31}
\hypersetup{
colorlinks=true,
linkcolor=amaranth,
citecolor=magenta,
urlcolor=cyan,
}

\newtheorem{thm}{Theorem}
\newtheorem{cor}[thm]{Corollary}
\newtheorem{lem}[thm]{Lemma}
\newtheorem{prop}[thm]{Proposition}

\theoremstyle{definition}
\newtheorem{defn}{Definition}
\newtheorem{exmp}[thm]{Example}

\newtheorem{rem}[thm]{Remark}

\newtheorem{innercustomgeneric}{\customgenericname}
\providecommand{\customgenericname}{}
\newcommand{\newcustomtheorem}[2]{%
  \newenvironment{#1}[1]
  {%
   \renewcommand\customgenericname{#2}%
   \renewcommand\theinnercustomgeneric{##1}%
   \innercustomgeneric
  }
  {\endinnercustomgeneric}
}

\newcustomtheorem{customthm}{Theorem}
\newcustomtheorem{customlem}{Lemma}
\newcustomtheorem{customprob}{Problem}
\newcustomtheorem{customcor}{Corollary}
\newcustomtheorem{customprop}{Proposition}
\newcustomtheorem{customexmp}{Example}
\newcustomtheorem{customdefn}{Definition}

\newcommand{\oc}[1]{\bot_{#1}}

\mathchardef\mhyphen="2D

\newcommand\restr[2]{{
  \left.\kern-\nulldelimiterspace 
  #1 
  \vphantom{\big|} 
  \right|_{#2} 
  }}

\DeclareMathOperator{\spn}{span}

\DeclareMathOperator{\ld}{d}

\def\oconverges{\xrightarrow{o}}

\newcommand{\s}{\mathcal S}
\newcommand{\p}{\mathcal P}

\newcommand{\f}{\mathcal F}

\newcommand{\el}{\mathcal L}
\newcommand{\e}{\mathcal E}
\newcommand{\ci}{\mathcal C}

\newcommand{\ff}{\mathscr F}
\newcommand{\uu}{\mathscr U}

\newcommand{\N}{\mathds N}
\newcommand{\R}{\mathds R}

\begin{document}

\title{Order topology on orthocomplemented posets of linear subspaces of a pre-Hilbert space}
\author{D. Buhagiar}
\address{
David Buhagiar \\
Department of Mathematics \\
Faculty of Science \\
University of Malta \\
Msida MSD 2080, Malta} \email {david.buhagiar@um.edu.mt}

\author{E. Chetcuti}
\address{
Emmanuel Chetcuti,
Department of Mathematics\\
Faculty of Science\\
University of Malta\\
Msida MSD 2080  Malta} \email {emanuel.chetcuti@um.edu.mt}

\author{H. Weber}
\address{Hans Weber\\
Dipartimento di Scienze matematiche, informatiche e fisiche\\
Universit\`a degli Studi di Udine\\
1-33100 Udine\\
Italia}
\email{hans.weber@uniud.it}

\date{\today}
\begin{abstract}
Motivated by the Hilbert-space model for quantum mechanics, we define a pre-Hilbert space logic to be a pair $(S,\el)$, where $S$ is a pre-Hilbert space and $\el$ is an
orthocomplemented poset of orthogonally closed linear subspaces of $S$, closed w.r.t. finite dimensional perturbations, (i.e. if $M\in\el$ and $F$ is a finite dimensional linear
subspace of $S$, then $M+F\in \el$).   We study the order topology $\tau_o(\el)$ on $\el$ and show that completeness of $S$ can by characterized
by the separation properties of the topological space $(\el,\tau_o(\el))$.  It will be seen that the
remarkable lack of a proper  probability-theory on pre-Hilbert space logics -- for an incomplete $S$ -- comes out elementarily from this topological characterization.
\end{abstract}
\maketitle

\section{Introduction}

This paper  contributes to the investigation of ``nonstandard
logics'' (see, for example, \cite{Pi2, Har1, Ho2, Har2, HoMaPlZu}) by considering orthocomplemented posets of linear subspaces of a pre-Hilbert space and
states (=probability measures) on them as ``quantum logics'' \cite{AmAr, GrKe, Ho1,  HaPt, DvPu, DvNePu, PtWe, BuCh,Tu}.

We define a pre-Hilbert space logic $(S,\el)$, and show that the (metric) completeness of $S$ can by characterized by the topological properties of $(\el,\tau_o(\el))$,
where $\tau_o(\el)$ is the order topology on $\el$.  It will also be shown that the
remarkable lack of a proper  probability-theory on pre-Hilbert space logics -- for an incomplete $S$ -- comes out elementarily from this topological characterization.

\subsection{The order topology on a poset}

Let $(P,\le)$ be a poset.  For any $a,b\in P$ let $[a,b]:=\{x\in P:a\le x\le b\}$, $[a,\rightarrow]:=\{x\in P:a\le x\}$ and $[\leftarrow,a]:=\{x\in P:x\le a\}$.
If $P$ is bounded,  its top and bottom elements are denoted by $1$ and $0$, respectively. For a subset $X$ of $P$ we shall write $\bigvee X$ and
$\bigwedge X$ for the supremum (resp. infimum) of $X$ (on the condition that they exist).
When the supremum (and resp. the infimum) of every bounded from above (resp. every bounded from below) subset of $P$ exists then
 $P$ is said to be \emph{Dedekind complete}. 
An element $a$ of a bounded poset $P$ is called an \emph{atom} if $0\neq a$ and $[0,a]=\{0,a\}$.   A bounded poset $(P,\le,0,1)$ is called an
\emph{orthocomplemented poset} if it has an orthocomplementation, i.e. a function $':P\to P$ satisfying {\rm(i)}~$a''=a$,
{\rm(ii)}~if $a\le b$, then $b'\le a'$, and {\rm(iii)}~$a\vee a'=1$, for every $a,b\in P$.
An \emph{ortholattice} is an orthocomplemented poset that is simultaneously a lattice.
An orthocomplemented poset satisfying that $a\vee b$ exists for every $a,b\in P$ such that $a\le b'$ is called an \emph{orthologic} and
an orthologic satisfying the orthomodular law
\[a\le b\ \Rightarrow\ b=a\vee(a'\wedge b)\]
is called an \emph{orthomodular logic}.

A net $\left(x_\gamma\right)_{\gamma\in\Gamma}$ is said to be increasing  if
$x_{\gamma_1}\le x_{\gamma_2}$ whenever $\gamma_1\leq \gamma_2$.
We shall denote by $x_\gamma\uparrow$ when $\left(x_\gamma\right)_{\gamma\in\Gamma}$ is increasing and if, in addition,
the supremum of $\left(x_\gamma\right)_{\gamma\in\Gamma}$ equals $x$, we shall write $x_\gamma\uparrow x$.
Dual symbols shall be used for decreasing nets.

A net $(x_\gamma)$ in $P$ is said to \emph{o-converge} to $x\in P$ if there are nets $(a_\gamma)$ and $(b_\gamma)$ in $P$
satisfying $a_\gamma\le x_\gamma\le b_\gamma$ for every $\gamma$ such that $a_\gamma\uparrow x$ and $b_\gamma\downarrow x$.
In this case we write $x_\gamma\oconverges x$ in $(P,\le)$.  Clearly, if $x_\gamma\uparrow x$ or $x_\gamma\downarrow x$ in $(P,\le)$,
then $x_\gamma\oconverges x$.  It is easy to see also that  the order limit of an o-convergent net is uniquely determined.  Note that
if $(P,\le)$ is Dedekind
complete, a net $(x_\gamma)_{\gamma\in\Gamma}$ order converges to $x$ in $(P,\le)$ if and only if
$\limsup_\gamma x_\gamma=\liminf_\gamma x_\gamma=x$.

A subset $X$ of $P$ is called \emph{o-closed} if no net  in $X$ o-converges to a point outside of $X$. The collection of all
o-closed sets comprises the closed sets for the \emph{order
topology} $\tau_o(P)$  of $P$.  The order topology of $P$ is the finest topology on $P$
that preserves o-convergence of nets; i.e. if $\tau$ is a topology on $P$ such that $x_\gamma\xrightarrow{o} x$ in $P$
implies $x_\gamma\xrightarrow{\tau} x$, then $\tau\subseteq \tau_o(P)$. Since  every o-convergent net is eventually  bounded, in the definition of o-closed
 sets it is enough to consider  bounded nets,
i.e.  a subset $X$ of $P$ is o-closed if and only if $X\cap[a,b]$ is o-closed for every $a,b\in P$.
Note that $\tau_o(P)$ is $T_1$ but in
general is not Hausdorff
\cite{Fl,FlKl}.

We recall that the \emph{interval topology} in  a poset $(P,\le)$ is the coarsest topology on $P$
such that every interval of the
type $[x,\rightarrow]$ or $[\leftarrow,x]$ is closed.  It is known that a bounded lattice is compact in its interval
topology if and only if  it
is Dedekind complete \cite[Theorem 20]{Birkhoff2}.

\begin{prop}\label{1.0} Let $(P,\le)$ be a poset.
\begin{enumerate}[{\rm(i)}]
\item The order topology on  $P$ is finer than the interval topology.
\item For  a monotonic net  $(x_\gamma)$ in $P$ and $x\in P$ the following conditions are equivalent:
\begin{enumerate}[{\rm(a)}]
\item either $ x_\gamma\uparrow x$ (when $x_\gamma\uparrow$) or $x_\gamma\downarrow x$ (when $x_\gamma\downarrow$).
\item  $x_\gamma\oconverges x$ in $P$.
\item $(x_\gamma)$ converges to $x$ w.r.t. $\tau_o(P)$.
\item $(x_\gamma)$ converges to $x$ w.r.t. the interval topology.
\end{enumerate}
\end{enumerate}
\end{prop}
\begin{proof}
{\rm(i)}~Consider e.g. a net $(x_\gamma)$ in $[a,\rightarrow]$ o-converging to $x$. Let $y_\gamma\leq x_\gamma\leq z_\gamma$ for all
$\gamma$ and $y_\gamma\uparrow x$, $z_\gamma\downarrow x$. Then $a\leq x_\gamma\leq z_\gamma\downarrow x$ implies $a\leq x$. Hence
$[a,\rightarrow]$ is o-closed.

{\rm(ii)}~Suppose that $x_\gamma\uparrow$.  $({\rm(a)}\Rightarrow{\rm(b)}\Rightarrow {\rm(c)}\Rightarrow {\rm(d)})$ is obvious.
$({\rm(d)}\Rightarrow {\rm(a)})$:  For any $\gamma_0\in\Gamma$ the
net $(x_\gamma)_{\gamma\geq\gamma_0}$ belongs to the closed
interval $[x_{\gamma_0},\rightarrow]$ and converges to $x$ w.r.t. the interval topology, hence $x\geq x_{\gamma_0}$. It follows that $x$ is an upper
bound of $\{x_\gamma:\gamma\in\Gamma\}$.  On the other hand, if $y$ is an
upper bound of $\{x_\gamma:\gamma\in\Gamma\}$, then $x\in[\leftarrow,y]$ since $[\leftarrow,y]$ is closed w.r.t. the interval topology. This
shows that $x=\bigvee_\gamma x_\gamma$.  The case when $x_\gamma\downarrow$ follows by duality.
\end{proof}




In general, if $P_0$ is a subset of a partially ordered set $P$, the topologies $\tau_o(P_0)$ and $\tau_o(P)\vert_{P_0}$ are not necessarily comparable.
This can elementarily be seen from the next example
\footnote{By taking the `sum' of these spaces, one gets an example for $\tau_o(P)|_{P_0}\nsubseteq \tau_o(P_0)\nsubseteq \tau_o(P)|_{P_0}$.}.

\begin{exmp} (i) $P=[0,1]\subseteq \R$, $P_0=\left\lbrace\frac{1}{2}-\frac{1}{n}:n\ge 2\right\rbrace\cup\lbrace 1\rbrace$. Then
$\tau_o(P)\vert_{P_0}$ is discrete and $\tau_o(P)\vert_{P_0}\nsubseteq\tau_o(P_0)$.

(ii) $P=\mathcal{P}(\mathds{N})$, $P_0=\lbrace A:\vert A\vert<\infty\rbrace$. Then $\tau_o(P_0)$ is discrete and
$\tau_o(P)\vert_{P_0}\nsupseteq\tau_o(P_0)$.
\end{exmp}

We shall, however, make use of the following observations.  We include the proofs for the sake of completeness.

\begin{prop}\label{1.6} Let $(P,\leq)$ be a partially ordered set and let $P_0$  be a subset of $P$.
\begin{enumerate}[{\rm(i)}]
\item If every increasing net in $P_0$ having a supremum in $P_0$ has the same supremum in $P$, and every decreasing net in $P_0$ having an infimum
in $P_0$ has the same infimum in $P$, then $\tau_o(P)|_{P_0}\subseteq \tau_o(P_0)$.
\item If $(P,\le)$ is Dedekind complete and $P_0$ is an o-closed sublattice
of $P$ then $\tau_o(P)|P_0= \tau_o(P_0)$.
\end{enumerate}
\end{prop}
\begin{proof} {\rm(i)} If a subset $X\subseteq P_0$ is not closed w.r.t. $\tau_o(P_0)$ one can find nets $(a_\gamma)$, $(b_\gamma)$ and
$(x_\gamma)$ in $P_0$ such that $x_\gamma\in X\cap [a_\gamma,b_\gamma]$ for every $\gamma$, $a_\gamma\uparrow$, $b_\gamma\downarrow$
and $\bigvee_\gamma a_\gamma=\bigwedge b_\gamma$ belongs not to $X$, where the supremum and infimum are here taken in $P_0$.
So the hypothesis precisely asserts that $(x_\gamma)$ is o-convergent to $x$ in the larger poset $P$, i.e. $X$ is not closed w.r.t. the subspace
topology $\tau_o(P)|_{P_0}$.  Hence $\tau_o(P)|_{P_0}\subseteq \tau_o(P_0)$.
{\rm(ii)} A net $(x_\gamma)$ in a Dedekind complete poset is o-convergent if and only if $\liminf_\gamma x_\gamma=\limsup_\gamma x_\gamma$.
Since $P_0$ is assumed to be an o-closed sublattice of $P$ it does not matter whether we take $\liminf$, $\limsup$ in $P$ or $P_0$ and so
the assertion follows.
\end{proof}

A function  $\varphi$ from a poset $(P,\le)$ into another poset $(Q,\le)$ is said to be:
\begin{enumerate}[{\rm(i)}]
  \item \emph{isotone} if $\varphi(a)\le\varphi(b)$ for every $a,b\in P$ satisfying $a\le b$;
 \item \emph{antitone} if $\varphi(b)\le\varphi(a)$ for every $a,b\in P$ satisfying $a\le b$;
  \item \emph{o-continuous} if $\varphi(x_\gamma)\oconverges \varphi(x)$ whenever $(x_\gamma)$ is a net in $P$ that o-converges to $x$.
\end{enumerate}
The posets $(P,\le)$ and $(Q,\le)$ are said to be \emph{order-isomorphic} (resp. \emph{dual order-isomorphic}) if there exists a bijection $\varphi$
between $P$ and $Q$ such that both $\varphi$ and $\varphi^{-1}$ are isotone (resp. antitone).

\begin{rem}\label{r:3} It is straightforward to verify that for a function $\varphi$ from a poset $(P,\le)$ into another poset $(Q,\le)$ the
following statements hold.
\begin{enumerate}[{\rm(i)}]
\item If $\varphi$ is isotone then $\varphi$ is o-continuous if and only if $f(x_\gamma)\uparrow f(x)$ (and resp. $f(x_\gamma)\downarrow f(x)$) whenever $x_\gamma\uparrow x$ (resp. $x_\gamma\downarrow x$).
\item If $\varphi$ is o-continuous then $\varphi$ is continuous w.r.t. $\tau_o(P)$ and $\tau_o(Q)$.  In particular, by {\rm(i)} and
Proposition \ref{1.0}~{\rm(ii)}, follows that for an isotone function, o-continuity is equivalent to continuity w.r.t. the respective order topologies.
We remark that for this equivalence the assumption on isotonicity of $\varphi$ is not redundant; to see this, one simply needs to take for
$(Q,\le)$ a poset admitting an unbounded (and therefore not o-convergent) sequence say $(y_n)$  that converges w.r.t. $\tau_o(Q)$ to some
point $y_\infty\in Q$, and then consider the function  $\varphi:\N\cup\{\infty\}\to Q$ defined by $\varphi(t):=y_t$.
\end{enumerate}
\end{rem}



The order topology on the projection lattice of a Hilbert space was studied in \cite{Pa} and more recently by the authors in \cite{BCW}.
A systematic treatment of the order topology associated to various structures of a von Neumann algebra was carried
out in \cite{ChHaWe}.
For example, it was shown that the order topology -- albeit far from being a linear topology --
on bounded subsets of  the self-adjoint part of a von Neumann algebra, coincides  with the strong operator topology.
Finite von Neumann algebras were also characterized by the order topological properties of their projection lattice.
The study of  the order topology induced by the star order on operator algebras was initiated in a very recent interesting work \cite{Bo}.

\subsection{Pre-Hilbert spaces}
Let $S$ be a pre-Hilbert space with inner product $\langle\cdot,\cdot\rangle$.  For a subset $A\subseteq S$ let
\[A^{\oc S}:=\{x\in S:\langle x,a\rangle=0\ \forall a\in A\}.\]
Then $A^{\oc S}$ is a closed linear subspace of $S$.  It is easily seen that the mapping $A\mapsto A^{\oc S\oc S}$ defines a Mackey closure
operation on $S$ (see, for example, \cite{Pi1}).  We shall call $A^{\oc S\oc S}$ the \emph{orthogonal closure of $A$ in $S$} and we shall say that a linear subspace $M$
of $S$ is \emph{orthogonally closed} if $M=M^{\oc S\oc S}$.

Let $M$ be a linear subspace of $S$.  Denote by $\ld(M)$  the linear or Hamel dimension of $M$.
A set $A$ of orthonormal vectors is said to be a \emph{maximal orthonormal system (MONS) in $M$} if $A\subseteq M$
and $\{0\}=A^{\oc S}\cap M=:A^{\oc M}$.  By Zorn's lemma,  $M$ has at least one MONS and by virtue of the Bessel inequality, one can show that any
two MONSs in $M$ have the same cardinality;  we call this cardinal number the \emph{orthogonal dimension of $M$} and denote it by $\dim M$.
If $A$ is an \emph{orthonormal basis (ONB) of $M$} then  $\dim M$ equals the \emph{topological density} of $M$.  In general, however, the orthogonal dimension of $M$
is strictly less than its topological density.

Let $\uu(S)$ denote the group of all isometric automorphisms of $S$.  This can be identified with the restriction to $S$ of all
the unitary operators $U$ on the completion $\overline S$ such that both $U$ and $U^\ast$ leave $S$ invariant.

\begin{prop}\label{unitary}
Let $\{x_1,\dots,x_n\}$ and $\{y_1,\dots,y_n\}$ be families of vectors in $S$ satisfying $\langle x_i,x_j\rangle=\langle y_i,y_j\rangle$ for every
$i,j\le n$.  Then there exists  $U\in \uu(S)$ such that $Ux_i=y_i$ for every $i\leq n$ and such that $U$ equals the identity operator on
$\{x_i,y_i:1\le i\le n\}^{\oc S}$.
\end{prop}
\begin{proof}
Let $X=\spn\{x_1,\dots,x_n\}$ and $Y=\spn\{y_1,\dots,y_n\}$. We may assume that $x_1,\dots,x_r$ is a base of $X$.
Let $T:X\rightarrow Y$ be the unique linear map such that $T(x_i)=y_i$ for $i\leq r$. One easily sees that $T$ preserves the
inner product: if $ u=\sum_{i=1}^r\alpha_ix_i$ and $v=\sum_{i=1}^r\beta_ix_i$ where $\alpha_i,\beta_i$ are scalars, then
$$\langle T(u),T(v)\rangle=\sum_{i,j\leq r}\alpha_i\overline{\beta_i}\langle y_i,y_j\rangle
=\sum_{i,j\leq r}\alpha_i\overline{\beta_i}\langle x_i,x_j\rangle=\langle u,v\rangle\,.  $$
Moreover, one has $T(x_j)=y_j$ also for $r<j\leq n$. In fact, for any $i\leq r$ one has
$\langle T(x_j),y_i\rangle=\langle T(x_j),T(x_i)\rangle=\langle x_j,x_i\rangle=\langle y_j,y_i\rangle$,
hence $\langle T(x_j),y\rangle=\langle y_j,y\rangle$ for all $y\in Y$ which implies $T(x_j)=y_j.$ In particular, $T$ is a linear isomorphism from
$X$ onto $T(X)=Y$ and $\ld(Y)=\ld(X)$.

We now extend $T$ to a unitary $U_0$ on $Z:=X+Y$. If $X=Y=Z$, let $U_0=T$.
Otherwise let $e_1,\dots,e_m$ be an ONB of $X^{\perp_Z}$ and $f_1,\dots,f_m$ an ONB of $Y^{\perp_Z}$. Then there is a unique
$U_0\in\uu(Z)$ such that $U_0$ extends $T$ and $U_0(e_i)=f_i$ for all $i\leq m$.

Finally, since $Z$ is finite dimensional and therefore  a splitting subspace of $S$, there is a unitary $U\in\uu(S)$ which extends $U_0$ and its
restriction to $Z^{\perp_S}$ is the identity operator.
\end{proof}


\section{Pre-Hilbert space logics}

In what follows $S$ is always a pre-Hilbert space.  For a non-zero vector $u\in S$ we write $[u]$ for the
one-dimensional linear subspace spanned by the vector $u$.
For an arbitrary subset $X$ of $S$ we denote by $\ff(X)$  the
upward directed family all finite dimensional linear subspaces spanned by vectors in $X$ and let $\tilde{\ff}(X):=\{A^{\oc{S}}:A\in\ff(X)\}$.

Motivated by the Hilbert-space model for quantum mechanics, which postulates that events of a quantum system are represented by closed
linear subspaces of a Hilbert space (see, for example, \cite{Va,PtPu})  we make the following definition.

\begin{defn} A \emph{pre-Hilbert space logic} is a pair $(S,\el)$ where $S$ is a pre-Hilbert space and $\el$ is a family of orthogonally-closed linear subspaces of
$S$ satisfying:
\begin{enumerate}[{\rm(i)}]
  \item $\{0\}\in\el$,
  \item if $M\in\el$ then $M^{\oc{S}}\in\el$,
  \item if $M\in \el$ and $u\in S$, then $M+[u]\in\el$.
\end{enumerate}
\end{defn}

The set-theoretic inclusion $\subseteq$ and the orthocomplementation operation $\perp:M\mapsto M^{\oc{S}}$ render a
pre-Hilbert space logic into an orthocomplemented poset with the atoms being the set of one-dimensional
linear subspaces of $S$.

Our definition of a pre-Hilbert space logic is molded on the following  well-studied  examples (see \cite{Dv1, Ha, BCD}) of  pre-Hilbert space logics:
Here we denote by $\mathscr V(S)$ the lattice of all linear subspaces of $S$.

\begin{itemize}
  \item $\f(S):=\{M\in \mathscr V(S):M=M^{\oc{S}\oc{S}}\}$ is the complete ortholattice of orthogonally closed linear subspaces of $S$.
  \item $\e_q(S):=\{M\in \mathscr V(S):M\text{ is closed and $M\oplus M^{\oc{S}}$ is dense in }S\}$ is the orthocomplemented poset of quasi-splitting linear subspaces of $S$.
  \item $\e(S):=\{M\in \mathscr V(S):M\oplus M^{\oc{S}}=S\}$ is the orthomodular poset of splitting linear subspaces of $S$.
  \item $\ci(S):=\{M\in \mathscr V(S): M \text{ is complete}\text{ or }M=N^{\oc{S}} \text{ for some complete $N\subseteq S$}\}$ is the orthomodular poset of complete/co-complete linear subspaces of $S$.
  \item $\p(S):=\{M\in \mathscr V(S):\dim M<\infty\text{ or }M=N^{\oc{S}}\text{ for some }N\subseteq S,\ \dim N<\infty\}$ is the modular lattice of
  finite/co-finite dimensional linear subspaces of $S$.
\end{itemize}

It is easily seen that these pre-Hilbert space logics are symmetric\footnote{
The pre-Hilbert space logic $(S,\el)$ is said to be \emph{symmetric} if  $U\el=\{UA:A\in\el\}\subseteq\el$ for every $U\in\uu(S)$.}
and that
$\p(S)\subseteq\ci(S)\subseteq \e(S)\subseteq \e_q(S)\subseteq \f(S)$.
With the  exception of $\p(S)$,  all these pre-Hilbert space logics are equal when $S$ is a Hilbert space.  In contrast, if, for example,
$S=c_{00}$ then
$$
\ci(S)\subsetneq \e(S)\subsetneq \e_q(S)\subsetneq \f(S).
$$
In fact, we have the following algebraic characterization of Hilbert spaces (see, for example, \cite{BCD}).

\begin{thm}
  For a pre-Hilbert space $S$ the following statements are equivalent:
  \begin{enumerate}[{\rm(i)}]
    \item $S$ is complete,
    \item $\f(S)$ is orthomodular,
    \item $S$ has an orthonormal basis and $\e_q(S)$ is orthomodular,
    \item $\e(S)$ is atomic $\sigma$-complete\footnote{We recall that an
orthocomplemented poset $(P,\le,0,1)$ is atomic $\sigma$-complete if whenever $(a_n)$ is a sequence of pairwise orthogonal atoms in $P$
the supremum $\bigvee_n a_n$ exists.}.
\end{enumerate}
\end{thm}

The definition of pre-Hilbert space logic allows us to have a unified approach: we shall study the order topology $\tau_o(\el)$ associated
with a generic pre-Hilbert space logic $(S,\el)$.

\begin{prop}\label{p:4} Let $(S,\el)$ be a pre-Hilbert space logic and let $\mathcal{X}\subseteq \el$.
\begin{enumerate}[{\rm(i)}]
\item Let $\tilde{\mathcal{X}}:=\{A^{\oc{S}}:A\in\mathcal{X}\}$.  Then $\vee \mathcal{X}$ exists in $\el$  if and only if
$\wedge\tilde{\mathcal{X}}$ exists in $\el$ and in this case $\vee \mathcal{X}=(\wedge\tilde{\mathcal{X}})^{\oc S}$.
\item $\wedge \mathcal X$ exists in $\el$  if and only if $\cap\mathcal X\in\el$ and in this case $\wedge\mathcal X=\cap\mathcal X$.
\item $\vee\mathcal X$ exists in $\el$ if and only if $(\cup \mathcal X)^{\oc{S}\oc{S}}\in\el$ and in this case
$\vee \mathcal X=(\cup\mathcal X)^{\oc{S}\oc{S}}$.
\end{enumerate}
\end{prop}
\begin{proof}
The first assertion follows by the fact that the operation $A\mapsto A^{\oc{S}}$ is a dual order-automorphism of $\el$.  {\rm (ii)}~If
$\cap\mathcal X\in\el$, then obviously $\cap\mathcal X$ is the infimum of $\mathcal{X}$ in $\el$.  If $\wedge\mathcal{X}$ exists in $\el$,
say is equal to $M$, then $M\subseteq \cap\mathcal{X}$.  If $u\in\cap\mathcal{X}$, then $M+[u]\subseteq \cap\mathcal{X}$ and
so $M+[u]\subseteq\wedge\mathcal{X}=M$, hence $u\in M$. It follows that $\cap\mathcal X=M\in\el$.  {\rm (iii)}~Using {\rm (i)}, {\rm (ii)}
and the equality $\cap\tilde{\mathcal{X}}=(\cup\mathcal{X})^{\oc{S}}$ one sees that the following conditions are equivalent:
$\vee\mathcal X$ exists, $\wedge\tilde{\mathcal{X}}$ exists,
$\cap\tilde{\mathcal{X}}\in\el$, $(\cup\mathcal{X})^{\perp_S}\in\el$, $(\cup\mathcal{X})^{\perp_S\perp_S}\in\el$.
Moreover, if $\vee\mathcal{X}$ exists, then
$\vee\mathcal{X}=(\wedge\tilde{\mathcal{X}})^{\oc{S}}=(\cap\tilde{\mathcal{X}})^{\oc{S}}=(\cup\mathcal{X})^{\oc S\oc S}.$


\end{proof}

\begin{lem}\label{l:6} Let $A\in \e(S)$.  Then $B^{\oc{A}\oc{A}}=B^{\oc{S}\oc{S}}$ for every $B\subseteq A$.
\end{lem}
\begin{proof}
It is clear that $B\subseteq A$ implies $B^{\oc{S}\oc{S}}\subseteq A^{\oc{S}\oc{S}}=A$, where the latter equality holds because $\e(S)\subseteq \f(S)$.  In addition, $B^{\oc{A}}=B^{\oc{S}}\cap A\subseteq B^{\oc{S}}$ implies that $B^{\oc{A}\oc{S}}{\supseteq} B^{\oc{S}\oc{S}}$.  So, it follows that $B^{\oc{S}\oc{S}}\subseteq B^{\oc{A}\oc{S}}\cap A=B^{\oc{A}\oc{A}}$.

We show that if $u\in B^{\oc{A}\oc{A}}$ then $u\bot B^{\oc{S}}$.  Let $u\in B^{\oc{A}\oc{A}}$, $v\in B^{\oc{S}}$ and $v=v_1+v_2$ be
the decomposition of $v$ such that $v_1\in A$ and $v_2\in A^{\oc{S}}$.  To prove $u\perp v$ we will show that $u\perp v_1$ and $u\perp v_2$.
We have $v_2\bot u$ because $u\in A$.  Moreover,
$v_2\in A^{\oc{S}}\subseteq B^{\oc{S}}$ and therefore $v_1=v-v_2\in A\cap B^{\oc{S}}=B^{\oc{A}}$.  So, $u\bot v_1$.
\end{proof}

A pre-Hilbert space logic $(U,\el_1)$ is called a \emph{pre-Hilbert space sublogic} of $(S,\el)$ if $U\subseteq S$ and $\el_1\subseteq \el$.  In the following theorem we characterize the splitting subspaces in terms of the subspace order topology.

\begin{thm}\label{t:1} Let $(U,\el_1)$ be a pre-Hilbert space sublogic of $(S,\el)$.  Then  $\tau_o(\el)|_{\el_1}\subseteq\tau_o(\el_1)$ if and only if $U$ is  splitting in $S$.
\end{thm}
\begin{proof}  ($\Rightarrow$) In view of  \cite[Theorem
  2.2]{BCD1} it suffices to show that every MONS $X$ of $U$
  satisfies $X^{\oc{S}\oc{S}}=U$. Let $X$ be a MONS of $U$. The upward directed family $\ff(X)$
  has a supremum equal to $U$ in $\el_1$ because
  $X^{\oc{U}\oc{U}}=U$.  So, $\ff(X)$ converges to $U$ w.r.t. $\tau_o(\el_1)$.  By hypothesis, it follows that
  $\ff(X)$ converges to $U$ w.r.t. $\tau_o(\el)$.   Propositions \ref{1.0}  and \ref{p:4} imply that
  $U=(\cup\{F:F\in\ff(X)\})^{\oc{S}\oc{S}}=X^{\oc{S}\oc{S}}$.

($\Leftarrow$)  By Proposition \ref{1.6} it suffices to show that if $A_\gamma\uparrow A$ and $B_\gamma\downarrow B$ in $\el_1$ then
$A_\gamma\uparrow A$ and $B_\gamma\downarrow B$ in $\el$. Let $A_\gamma\uparrow A$ in $\el_1$. Then by Proposition \ref{p:4} and Lemma \ref{l:6}
\[A=\left(\bigcup_{\gamma}A_\gamma\right)^{\perp_U\perp_U}=\left(\bigcup_{\gamma}A_\gamma\right)^{\perp_S\perp_S}, \]
and so $A_\gamma\uparrow A$ also in $\el$.

If  $B_\gamma\downarrow B$ in $\el_1$, then by
Proposition \ref{p:4}, $\bigcap_{\gamma}B_\gamma=B$  and so $B_\gamma\downarrow B$ in $\el$.

\end{proof}

\begin{cor}\label{c:5}
Let $\el_1$ and $\el_2$ be two pre-Hilbert space logics associated with $S$ such that $\el_1\subseteq \el_2$.  Then $\tau_o(\el_2)|_{\el_1}\subseteq \tau_o(\el_1)$.
\end{cor}
\begin{proof}
This follows by the theorem because $S\in \e(S)$.
\end{proof}



By definition, a pre-Hilbert space logic is closed w.r.t. finite dimensional perturbations.  The following proposition asserts that finite dimensional perturbations  are o-continuous.  This will be needed in the proof of the main theorem.

\begin{prop}\label{p:1}
Let $(S,\el)$ be a pre-Hilbert space logic.  The mapping $T_F:\el\to\el\ :\ A\mapsto A+F$ is o-continuous for every finite dimensional linear subspace $F$ of $S$.
\end{prop}
\begin{proof}
It is enough to prove the proposition in the case when $F=[u]$ for some $u\in S$.  The proof then follows by a straightforward inductive argument.
By Remark \ref{r:3} it suffices to show that if $A_\gamma\uparrow A$ and $B_\gamma\downarrow B$ in $\el$ then $A_\gamma+[u]\uparrow A+[u]$ and
 $B_\gamma+[u]\downarrow B+[u]$.  Clearly, $A+[u]\supseteq A_\gamma + [u]$ for every $\gamma$ and if
 $M\in \el$ satisfies $M\supseteq A_\gamma+[u]\supseteq A_\gamma$ for every $\gamma$, then $M\supseteq A$.
 Since, obviously, $M\supseteq [u]$, it follows that $M\supseteq A+[u]$.  For the second assertion we can suppose that eventually $u\notin B_\gamma$,
 i.e. we can in fact suppose that $u\notin B_\gamma$ for every $\gamma$.  By Proposition \ref{p:4} we have
 $ B=\bigcap_\gamma B_\gamma$, and it is enough to show that
 $ B+[u]=\bigcap_\gamma (B_\gamma+[u])$. The inclusion $\subseteq$ is obvious. On the other-hand, if $x\in B_\gamma+[u]$ for every
  $\gamma$, then for $\gamma'\ge\gamma$ there are $b_\gamma\in B_\gamma$, $b_{\gamma'}\in B_{\gamma'}$ and scalars $\lambda_{\gamma}$,
   $\lambda_{\gamma'}$ such that $x=b_\gamma+\lambda_\gamma u=b_{\gamma'}+\lambda_{\gamma'} u$, hence
   $b_\gamma-b_{\gamma'}=(\lambda_{\gamma'}-\lambda_{\gamma})u$ and so, since $u\notin B_{\gamma}$ we get that $b_\gamma=b_{\gamma'}$ and
   $\lambda_{\gamma}=\lambda_{\gamma'}$. Thus there exists a scalar $\lambda$ such that $x-\lambda u\in \bigcap_\gamma B_\gamma$. This shows,
   using Proposition \ref{p:4}, that $x\in \left(\bigcap_\gamma B_\gamma\right)+[u]=B+[u]$..
\end{proof}

\begin{rem}\label{r:4}
 If in Proposition \ref{p:1}, $F$ has infinite dimension, then $T_F$ is in general not anymore o-continuous:
 let $e_1,e_2, \dots$ be an orthonormal base of $\ell_2$ and $F:=[e]^\perp$ where $e=\sum_{n=1}^\infty\frac{1}{n}e_n$.
 Then $A_n:=\{e_1,\dots,e_n\}^\perp\downarrow 0$, but $A_n+F=\ell_2$; hence $A_n+F\downarrow F$ is not true.
\end{rem}

\begin{prop}\label{p:10}
Let $(S,\el)$ be a pre-Hilbert space logic.  If  $X$
is a dense subset of $S$ then $\ff(X)$ is dense in $\el$ w.r.t. $\tau_o(\el)$.
\end{prop}

The proof of the proposition relies on the following algebraic lemma.  It is  known for pre-Hilbert spaces
but for the sake of completeness we prove a more general version of \cite[Theorem 34.9]{MaMa}.

\begin{lem}\label{p:9}
Let $E$ be a Hausdorff topological vector space, $V$ a dense linear subspace
of $E$, and $N$ a closed linear subspace of $E$ such that $d(E/N)<\infty$.
Then $V\cap N$ is dense in $N$.
\end{lem}
\begin{proof}
Let $F$ be an algebraic complement of $N$ in $E$. Since $d(F)<\infty$, $F$ is also a topological complement.

{\rm(i)} To prove that $V\cap N$ is dense in $N$, we first consider the case that $d(F)=1$.
Since by the assumptions $V$ cannot be contained in $N$, there are $x_0\in N$ and $f\in F\setminus\{0\}$ such that $v_0:=x_0+f\in V$.
Let $x\in N$ and $v_\gamma\in V$ with $v_\gamma\rightarrow x$. Let $x_\gamma\in N$ and $\lambda_\gamma$ be
scalars with $v_\gamma=x_\gamma+\lambda_\gamma f$. Then $x_\gamma\rightarrow x$ and $\lambda_\gamma\rightarrow 0$.
Therefore $ v_\gamma-\lambda_\gamma v_0=x_\gamma-\lambda_\gamma x_0\in V\cap N$ and converges to $x$.
This proves that $V\cap N$ is dense in $N$.

{\rm(ii)} Let now $d(F)=m$ and  $N_i$ be closed linear subspaces of $E$ such that
$E=N_0\supset N_1\supset\dots\supset N_m=N$ and $d(N_{i-1}/N_i)=1$ for $i=1,\dots m$.
Inductively for $i=0,\dots m$ one sees using {\rm (i)} that $V_i:=V\cap N_i$ is dense in $N_i$.
Hence $\overline{V\cap N}=\overline{V_m}=N_m=N$.
\end{proof}

\begin{proof} [Proof of Proposition \ref{p:10}]  For the proof we can replace $X$ with $U:=\spn X$.

{\rm (i)} Let $M\in P(S)$, $\dim M^{\oc{S}}<\infty$.  By Lemma \ref{p:9}  $M\cap U$ is dense in $M$ hence $(M\cap U)^{\oc{S}\oc{S}}=M^{\oc{S}\oc{S}}=M$.
 Therefore $\mathscr F(M\cap U)\uparrow (M\cap U)^{\oc{S}\oc{S}}=M$.

{\rm(ii)} Let us show that for every $A\in\el$ the net $\tilde{\mathscr F}(A^{\oc{S}})$ decreases
to $A$ in $\el$. Clearly, $\mathscr F(A^{\oc{S}})\uparrow A^{\oc{S}}$, hence $\tilde{\mathscr F}(A^{\oc{S}})\downarrow A^{\oc{S}\oc{S}}=A$
(by Proposition \ref{p:4}).

{\rm (iii)} By {\rm(i)} $\mathscr F(U)$ is dense in $\tilde{\mathscr F}(S)$ w.r.t. $\tau_o(\el)$.
Since $\tilde{\mathscr F}(A^{\oc{S}})\subseteq \tilde{\mathscr F}(S)$ for $A\in \el$, {\rm (ii)} implies that $\tilde{\mathscr F}(S)$ is
dense in $(\el,\tau_o(\el))$. These two facts imply the thesis.
\end{proof}

A \emph{state} on the pre-Hilbert space logic $(S,\el)$ is a normalized and positive additive function on $\el$, i.e. a function $s:\el\to[0,1]$
satisfying $s(S)=1$ and
\begin{equation}\label{e:4}
s(A_1\vee A_2)=s(A_1)+s(A_2)\text{ whenever }A_1,A_2, A_1\vee A_2\in\el\text{ and }A_1\bot A_2.
\end{equation}
Let us denote by $(S,\el)_*$ the set of o-continuous states on $(S,\el)$.
\begin{rem}\label{r:5}
Every state $s$  on $(S,\el)$ is isotone and satisfies $s(A)+s(A^\perp)=1$ for any $A\in\el$; so a state
$s$ is $o$-continuous if and only if $s(A_\gamma)\uparrow s(A)$ whenever $A_\gamma\uparrow A$ in $\el$.
\end{rem}



We now present the main result of the paper.

\begin{thm}[Main Theorem]\label{thm}  Let $(S,\el)$ be a pre-Hilbert space logic such that $\dim S=\aleph_0$.  The following are equivalent.
\begin{enumerate}[{\rm(i)}]
  \item $\tau_o(\el)$ is a Hausdorff topology.
  \item $S$ is a Hilbert space.
  \item $(S,\el)_*\neq \emptyset$.
\end{enumerate}
\end{thm}

\section{Proof of Main Theorem}

Let $(S,\el)$ be a pre-Hilbert space logic, where $S$ is any infinite dimensional pre-Hilbert space.    As we go along the proof, we shall explicitly indicate
where the assumption $\dim S=\aleph_0$ is needed; it will become clear that the implications ${\rm(ii)}\Rightarrow {\rm(iii)}\Rightarrow {\rm(i)}$ hold
without this assumption.

When $(S,\el)$ is a symmetric pre-Hilbert space logic, $s$ is a state on $(S,\el)$ and $U\in\uu(S)$,
the function $s^U:A\mapsto s(UA)$ defines a state on $\el$. In addition, if $s$ is o-continuous then so is $s^U$.
We define $\langle s\rangle:=\{s^U:U\in\uu(S)\}$.

\subsection*{To show {\rm(i)}$\Rightarrow${\rm(ii)}.}
\begin{lem}\label{l:1} Let $A,B$ be linear subspaces of a vector space $V$ and let $x$, $y$ be vectors in $V$ such that $x+A\subseteq y+B$.
Then $x+B = y+B$, $A\subseteq B$ and $[x]+A\subseteq [y]+B$.
\end{lem}
\begin{proof} Since $x = y+b$ for some $b\in B$ it follows that
\[
x+B = y+b+B = y+B \supseteq x +A\,,
\]
and therefore $A\subseteq B$.  The last statement follows since $x\in x+B = y+B \subseteq [y]+B$ and $A\subseteq B\subseteq [y]+B$.
\end{proof}

\begin{lem}\label{l:2} Let $V$ be a vector space. Let $(A_\gamma)$ be a net of linear subspaces of $V$ and $(x_\gamma)$ a net in $V$ such
that $(x_\gamma + A_\gamma)$ is decreasing, and $\bigcap_\gamma x_\gamma +
  A_\gamma\subseteq \bigcap_\gamma A_\gamma$. Then $\bigl([x_\gamma] +
  A_\gamma\bigr)$ is decreasing and $\bigcap_\gamma \bigl([x_\gamma] +
  A_\gamma\bigr) =\bigcap_\gamma A_\gamma$.
\end{lem}
\begin{proof} The nets $(A_\gamma)$ and  $\bigl([x_\gamma] +
  A_\gamma\bigr)$ are decreasing by Lemma \ref{l:1}.
Seeking a contradiction, let $x\in \bigcap_\gamma \bigl([x_\gamma] +
  A_\gamma\bigr)\setminus \bigcap_\gamma A_\gamma$. One may assume
  that $x\notin A_\gamma$ for every $\gamma$. Then $x\in ([x_\gamma] +
  A_\gamma) \setminus A_\gamma$ for every $\gamma$, and therefore $x = \lambda'_\gamma
  x_\gamma + a'_\gamma$, where $\lambda'_\gamma \neq 0$ and $a'_\gamma
  \in A_\gamma$. Rearranging, we get $x_\gamma = \lambda_\gamma
  x + a_\gamma$, where $\lambda_\gamma \neq 0$ and $a_\gamma
  \in A_\gamma$. Thus, for $\gamma_1\leq \gamma_2$,
\[
\lambda_{\gamma_2}x \in x_{\gamma_2} + A_{\gamma_2} \subseteq
x_{\gamma_1} + A_{\gamma_1} =  \lambda_{\gamma_1}x + A_{\gamma_1}
\]
implying that $(\lambda_{\gamma_2} - \lambda_{\gamma_1})x \in
A_{\gamma_1}$, and therefore $\lambda_{\gamma_1} =
\lambda_{\gamma_2}$. Consequently, $\lambda_{\gamma_1}x =
\lambda_{\gamma_2}x \in x_{\gamma_2} + A_{\gamma_2}$ and this leads to
the following contradiction:
\[
\lambda_{\gamma_1}x \in \bigcap_{\gamma\geq\gamma_1} x_\gamma +
A_\gamma = \bigcap_{\gamma} x_\gamma +
A_\gamma\subseteq \bigcap_{\gamma} A_\gamma\,.
\]
\end{proof}

\begin{prop}\label{p:11}
A pre-Hilbert space $S$ is incomplete if and only if  there exists a net $(x_\gamma)$ in the unit ball of $S$ such that
$[x_\gamma+y]{\overset{o}{\longrightarrow}}[0]$
in $\p(S)$ for every $y\in S$.
\end{prop}
\begin{proof}
($\Leftarrow$)  Let $(x_\gamma)_{\gamma\in\Gamma}$ be a net in the unit ball of $S$ such
that $[x_\gamma+y]{\overset{o}{\longrightarrow}}[0]$ in $\mathcal{P}(S)$ for every $y\in S$. Suppose that $S$ is complete.
Then the unit ball of $S$ is weakly compact and therefore $(x_\gamma)$ has a weak cluster point, say $x$. Let $y\in S\setminus\{-x\}$ and
$(A_\gamma)$ a net in $\mathcal{P}(S)$ decreasing to $[0]$ satisfying $[x_\gamma+y]\subseteq A_\gamma$ for all $\gamma\in\Gamma$.
Then for fixed $\gamma\in\Gamma$ one has $x_{\gamma'}+y\in A_\gamma$ for all $\gamma'\geq\gamma$, hence $x+y\in A_\gamma$ since
$A_\gamma$ is closed, hence weakly closed. This shows that $x+y\in\bigcap_\gamma A_\gamma=\{0\},$ hence $x+y=0$, a contradiction.


($\Rightarrow$)   Let $H$ denote the completion of $S$ and let $u\in H\setminus S$. We first show, using Lemma \ref{l:2}, that the net
$\bigl([P_Fu]+F^{\perp_S}\bigr)_{F\in\mathscr F(S)}$ decreases to $[0]$. It is clear that this net is decreasing.  In addition, if $x\in P_Fu+F^{\oc{S}}$
then $\langle x,f\rangle=\langle P_Fu,f\rangle=\langle u,f\rangle$ for every $f\in F$, and therefore, if $x\in\bigcap_{F\in\mathscr F(S)}P_Fu+F^{\oc{S}}$
then $\langle x,f\rangle=\langle u,f\rangle$ for every $f\in S$, implying that $x=u\in H\setminus S$, a contradiction.
Therefore $\bigcap_{F\in\mathscr F(S)}P_Fu+F^{\oc{S}}=\emptyset$, hence
$[P_Fu]+F^{\oc{S}}\downarrow [0]$ by Lemma \ref{l:2}.


Let us verify that the net $(x_F)_{F\in\mathscr F(X)}$ where $x_F:=P_Fv$ and $v$ is a unit vector in $H\setminus S$ satisfies the required properties
to prove the assertion.  Clearly $\Vert x_F\Vert \le \Vert v\Vert =1$.  The vector $v+y$ is not in $S$ when $y\in S$ and therefore, as proved before,
the linear subspaces $\bigl[P_F(v+y)\bigr]+F^{\oc{S}}$ form a net in $\p(S)$ that decreases to $[0]$.  Since
$[x_F+y]\subseteq \bigl[P_F(v+y)\bigr]+F^{\oc{S}}$ it follows that $[x_F+y]{\overset{o}{\longrightarrow}}[0]$ in $\p(S)$.
\end{proof}

\begin{rem}\label{c:12}
So, by Proposition \ref{p:1}, if we suppose that $S$ is incomplete then there exists a net $(x_\gamma)$ in the unit ball of $S$ such that
$[x_\gamma+y]+F{\overset{\tau_o(\el)}{\longrightarrow}}F$ for every $y\in S$ and every finite dimensional linear subspace $F\subseteq S$.
\end{rem}

\begin{lem}\label{l:13}  Suppose that $S$ is incomplete and let $\mathcal U\subseteq \el$ be open w.r.t. $\tau_o(\el)$.
\begin{enumerate}[{\rm(i)}]
\item Let $U\in \mathcal U$ be finite dimensional and  $\{x_1,\dots,x_k\}$ be a finite set of vectors in $S$. For every $\varepsilon >0$, there exists
finite dimensional $V\in\mathcal U$ such that $U\subseteq V$ and $d(x_i,V):=\inf\{\Vert x_i-v\Vert:v\in V\} <\varepsilon$ for every $1\leq i\leq k$.
\item Suppose that $\mathcal U\neq \emptyset$ and let $X$ be a countable set of vectors in $S$.  Then there exists  an increasing sequence
$(V_n)$ of finite dimensional linear subspaces in $\mathcal U$ such that  $X\subseteq \mbox{cl}_S\left({\bigcup_{n} V_n}\right)$.
\end{enumerate}
\end{lem}
\begin{proof}  {\rm(i)}~Fix $\varepsilon>0$.  Let  $n\in\mathds N$ such that $n>1/\varepsilon$.  By Remark \ref{c:12} we can find a net
$(x_\gamma)$ in the unit ball of $S$ such that $[x_\gamma+y]+F\overset{\tau_o(\el)}{\longrightarrow}F$ for every $y\in S$
and finite dimensional linear subspace
$F\subseteq S$.    A straightforward induction shows that there exist $\gamma_1\le\gamma_2\le\cdots\le\gamma_k$  such that
\[V:=[x_{\gamma_k}+nx_k]+[x_{\gamma_{k-1}}+nx_{k-1}+\cdots+[x_{\gamma_{1}}+nx_1]+U\in \mathcal U.\]
Then $x_i+(x_{\gamma_{i}}/n)\in V$ and so $d(x_i,V)\le1/n<\varepsilon$ for every $i\le n$.

{\rm(ii)}~Given $M\in\mathcal U$, the net $\mathcal F(M)$ is eventually in $\mathcal U$ and therefore there exists a finite dimensional  linear subspace
$V_0$ that is an element of $\mathcal U$.  Let $X=\{x_i:i\in\N\}$.  One can apply {\rm(i)} to construct, by induction, an increasing sequence $(V_n)$ of
finite dimensional linear subspaces of $S$, contained in $\mathcal U$, such that $V_0\subseteq V_1$, and $d(x_i,V_n)<1/n$ for all $i\le n$.
\end{proof}

Let us now proceed to the proof {\rm(i)}$\Rightarrow${\rm(ii)}.  Here we need the assumption that $\dim S=\aleph_0$ in order to apply
{\rm(ii)} of Lemma \ref{l:13} with $X$ equal to a (countable) MONS of $S$, to obtain that for every nonempty open set $\mathcal U\subseteq \el$
there exists an increasing sequence $(V_n)$ of finite dimensional linear subspaces, contained in $\mathcal U$, such that
\[S=X^{\oc{S}\oc{S}}\subseteq \bigl(\mbox{cl}_S\bigl(\bigcup_n V_n\bigr)\bigr)^{{\oc S}{\oc S}}=\bigl(\bigcup_n V_n\bigr)^{{\oc S}{\oc S}}=\bigvee_n V_n,\]
i.e. $S$ belongs to the closure of any nonempty open set $\mathcal U\subseteq \el$.


\subsection*{To show {\rm(ii)}$\Rightarrow${\rm(iii)} and {\rm(ii)}$\Rightarrow${\rm(i)}.}

Let $S$ be a Hilbert space. As known, every unit vector $u\in S$ induces the state
\[s_u:\el\to[0,1]:M\mapsto\Vert P_M u\Vert^2,\]
where $P_M$ denotes the projection of $S$ onto $M$.

To see that $s_u$ is o-continuous, in view of Remark \ref{r:5}, one simply needs to recall that when $M_\gamma,M\in\el$ then
$M_\gamma\uparrow M \text{ in }\el$ implies, since $S$ is complete, $M=\left(\bigcup_\gamma M_\gamma\right)^{\oc S\oc S}=
\overline{\bigcup_\gamma M_\gamma}\,,$ hence $\Vert P_{M_\gamma}v\Vert\uparrow\Vert P_M v\Vert$.


It is clear that $\{s_u:u\in S,\,\Vert u\Vert=1\}$ is a separating subset of $(S,\el)_\ast$ and
thus $(\el,\tau_o(\el))$ is (even) functionally Hausdorff.


\subsection*{To show {\rm(iii)}$\Rightarrow${\rm(ii)}.}  Let $s\in(S,\el)_\ast$.  Then, by Corollary \ref{c:5}, the restriction of $s$ to $\p(S)$ belongs to
$(S,\p(S))_\ast$.  Since $\ff(S)\uparrow S$ in $\p(S)$ and $s$ is o-continuous, we have $s(F)\neq 0$ for some finite dimensional linear subspace $F$ of $S$.
Therefore, by the next proposition $(\p(S),\tau_o(\p(S)))$ is functionally Hausdorff and so $S$ is complete because we know
that {\rm(i)} is equivalent to {\rm(ii)}.

\begin{prop}\label{p:13}  Let $s$ be a state on $\p(S)$. The following conditions are equivalent.
\begin{enumerate}[{\rm(i)}]
\item $s$ is not the state assigning $0$ to all the finite dimensional linear subspaces of $S$.
\item $s([x])\neq s([y]) \text{ for some unit vectors }x,y\in S$.
\item For any pair of linearly independent vectors $a,u\in S$ there exists $U\in \uu(S)$ such that $s^U([a])<s^U([u])$.
\item $\langle s\rangle$ separates $\p(S)$.
\end{enumerate}
\end{prop}
\begin{proof}
$({\rm (iv)}\Rightarrow {\rm (i)})$ is trivial.

$({\rm (i)}\Rightarrow {\rm (ii)})$. First note that if $F\in\p(S)$ is finite dimensional and $s(F)\ne 0$ one can find a unit vector $u$ in $F$ with
$s([u])>0$.  Since $\sum_{k=1}^{\infty}s([f_k])\le 1$ for
every orthonormal system $(f_k)$ in $S$, one can find a unit vector $v$ with $s([v])<s([u])$.

$({\rm (iii)}\Rightarrow {\rm (iv)})$. Let $A,B\in \mathcal{P}(S)$ and $A\neq B$. Obviously $\langle s\rangle$ separates $A$ and $B$ if
and only if $\langle s\rangle$ separates $A^{\perp_S}$ and $B^{\perp_S}$. Therefore we may assume that $A$ is finite dimensional and
$B\nsubseteq A$. Let $u\in B\setminus A$. Since $A$ is splitting, there are $a\in A$ and $b\in A^{\oc S}$ such
 that $u=a+b$.
By ${\rm(iii)}$ there is $V\in \uu(S)$ such that $s^{V}([a])< s^{V}([u])$.
 Let $A_0:=[a]^{\oc{A}}$ and let $K:=\spn\{u,a\}$.  Since $\sum_{k=1}^{\infty}s([f_k])\le 1$
 for every orthonormal system $(f_k)$ in $K^{\oc{S}}$ there is a linear
subspace $B_0\subseteq K^{\oc{S}}$ such that $\dim B_0=\dim A_0$ and $s^V(B_0)<s^V([u])-s^V([a])$.  By Proposition \ref{unitary} there is a
$W\in \uu(S)$ such that $WA_0=B_0$ and $W$ is the identity on $K$.  Then $U:=VW\in\uu(S)$ and
\[s^U(A)=s^U([a])+s^U(A_0)=s^V([a])+s^V(B_0)< s^V([u])=s^U([u])\leq s^U(B).\]

$({\rm (ii)}\Rightarrow {\rm (iii)})$.  For this we shall need the following lemma.

\begin{lem}\label{chain}
 Let $V$ be a pre-Hilbert space with real linear dimension at least 3,  and let $x\ne y$ be unit vectors in $V$.  For every $0<\varepsilon<2$ there are
unit vectors $e_0=x,e_1,\dots,e_{n+1}=y$ in $V$ such that $\Vert e_i-e_{i-1}\Vert=\varepsilon$ for every $i=1,\dots,n+1$.\\
Consequently, if $f$ is a non-constant function on the unit sphere of $V$ and $0<\varepsilon<2$ then
there are points $x$ and $y$ on the unit sphere such that $\Vert x-y\Vert=\varepsilon$ and $f(x)< f(y)$.
\end{lem}
\begin{proof}
Let us call a finite sequence $(x_0,\dots,x_n)$ in a metric space $(X,\rho)$ a \emph{strict $\varepsilon$-chain} when $\rho(x_i,x_{i-1})=\varepsilon$
for every $i$; in this case we say that $x_0$ and $x_n$ can be connected by a strict $\varepsilon$-chain.  If any two points of $X$ can be connected by a
strict $\varepsilon$-chain, then $X$ is said to be strictly $\varepsilon$-chain connected. We have to prove that the unit sphere of the pre-Hilbert space
$V$ is strict $\varepsilon$-chain connected.  Clearly, it is enough to establish the claim in the case when $\dim V<\infty$.  Since every finite dimensional
Hilbert space is isometric to $\ell_2^k(\R)$, with $k$ equal to the real linear dimension of $V$, it suffices to establish the assertion when $V=\R^3$.
\footnote{Note that this is the minimal value of $k$ for which this is true: indeed, every point on the unit circle is strictly $\varepsilon$-chain connected
to a countable number of points; i.e. the unit sphere of $\R^2$ cannot be strictly $\varepsilon$-connected.}
 Let $\s^2$ denote the unit sphere in $\R^3$ and let $\theta$ denote the circular arc-metric on $\s^2$.  Due to the nature of our claim, it can be observed that it is enough to
 establish the assertion for the metric space $(\s^2,\theta)$ and $0<\varepsilon<\pi$.  For any $v\in\s^2$ and $0<\varepsilon<\pi$ let
 $C(v,\varepsilon):=\{u\in\s^2:\theta(u,v)=\varepsilon\}$.  Observe that $C(v,\varepsilon)=C(-v,\pi-\varepsilon)$ and so if $(x_0,x_1,x_2,\dots,x_n)$ is a
 strict $\varepsilon$-chain, then $(x_0,-x_1,x_2,\dots,(-1)^nx_n)$ is a strict $(\pi-\varepsilon)$-chain.  Fix $x,y\in\s^2$ and let $\delta:=\theta(x,y)$.
 Let $n:=\max\{i\in\N:(i-1)\varepsilon<\delta\}$.  The range of the continuous function $u\mapsto \theta(x,u)$ when restricted to the arc segment $C$
 joining $x$ to $y$ which has minimal length, is $[0,\delta]$.  So there are points $x_0=x, x_1,\dots, x_{n-1}$ on $C$ such that
 $\theta(x,x_{i})=i\varepsilon$ for every $i=0,\dots,n-1$.  So $\theta(x_{i+1},x_i)=\varepsilon$ for every $i<n-2$.  First suppose that
 $0< \varepsilon\le \pi/2$.  Maximality of $n$ together with the assumption $\varepsilon\le\pi/2$ yield that $C({x_{n-1}},\varepsilon)$ and $C(y,\varepsilon)$
 have a point of intersection (in fact, precisely two points).  Let $x_n$ be any one of these points. Let $u$ be a point on $C(y,\varepsilon)$ such that
 $\theta(x_n,u)=\varepsilon$.   Then $(x_0,\dots,x_{n},y)$ and $(x_0,\dots,x_n,u,y)$ are two strict $\varepsilon$-chains in $\s^2$ joining $x$ to $y$.
 If we suppose that $\pi/2 < \varepsilon<\pi$ then take the two strict $\varepsilon'$-chains in $\s^2$ joining $x$ to $y$ with $\varepsilon':=\pi-\varepsilon$,
 as defined above, and choose the one with an odd number of points.  Let us express it as $(e_0,e_1,\dots,e_{2m})$.  Then $(e_0,-e_1,e_2,\dots,e_{2m})$ is
 the required strict $\varepsilon$-chain joining $x$ to $y$.
\end{proof}

Let $u$ and $a$ be linearly independent unit vectors in $S$ and let $\alpha\in\mathbb{C}$, $|\alpha|=1$
satisfy $\langle\alpha u,a\rangle =|\langle u,a\rangle|$. Since $[\alpha u]=[u]$ we may
assume, replacing $u$ by $\alpha u$, that $\langle u,a\rangle\geq 0$.
Let $v$ be a unit vector in $\spn\{a,u\}$ such that $u\bot v$ and $\langle a,v\rangle$ is real.
We define a unit vector $w\in \{u,v\}^{\oc S}$  such that $s([u])$, $s([v])$ and $s([w])$ are not all equal as follows:\\
\noindent\emph{Case 1} If $s(\spn\{u,v\})=0$, let $x\in S$ be a unit vector such that $s([x])>0$ and let $w\in\spn\{u,v,x\}$ be a unit vector
orthogonal to $u$ and
$v$.  Note that $s([w])=s([u])+s([v])+s([w])=s(\spn\{u,v,w\})\ge s([x])>0$.\\
\noindent\emph{Case 2} If $s(\spn\{u,v\})\ne 0$ -- since $\sum_{k=1}^{\infty}s([f_k])\le 1$ for every orthonormal system
$(f_k)$ in $\{u,v\}^{\oc{S}}$ -- there is a unit vector $w\in \{u,v\}^{\oc S}$ such that $s([w])<\max\{s([u]),s([v])\}$.\\
Observe that restriction of the scalar product on the three-dimensional real linear space
$V:=\{\alpha u+\beta v+\gamma w\,:\,\alpha,\beta,\gamma\in\R\}$ is real-valued and that $a=\langle a,u\rangle u+\langle a,v\rangle v\in V$.
Moreover, by the choice of $w$ the function on the unit sphere of $V$ defined by $f(z):=s([z])$ is non-constant and so,
by Lemma \ref{chain}, there are points $x$ and $y$ on the unit sphere of $V$ such that $\Vert x-y\Vert=\Vert u-a\Vert$ and $f(x)< f(y)$.

From
\[2\langle u,a\rangle=\Vert u\Vert +\Vert a\Vert -\Vert u-a\Vert^2=\Vert x\Vert +\Vert y\Vert -\Vert x-y\Vert^2=2\langle x,y\rangle\]
we deduce that $\langle u,a\rangle=\langle x,y\rangle$ and so by Proposition \ref{unitary} there is a unitary operator $U\in\uu(S)$ such that $Uu=x$ and
$Ua=y$.
Then
\[s^U([u])=s([Uu])=s([x])< s([y])=s([Ua])=s^U([a]).\]
\end{proof}
So the proof of the Main Theorem is complete.

\section{Remarks}

\begin{rem} In the proof of Proposition \ref{p:13} we were inclined towards elementariness.  It should be noted, however, that the proof of
$({\rm (ii)}\Rightarrow {\rm (iii)})$
can be simplified if we employ Gleason's Theorem (three-dimensional version).
Let us outline the argument.  Let $a$ and $u$ be linearly independent unit vectors in $S$.  Let $v\in\spn\{u,a\}$ be a unit vector orthogonal with $u$
and let $w\in \{u,v\}^{\oc{S}}$ be a unit vector satisfying that $s([u])$, $s([v])$ and $s([w])$ are not all equal (as explained above).  Let $M:=\spn\{u,v,w\}$.
Then $s'(A):=(s(M))^{-1}\,s(A)$ defines a state on the projection lattice of the three-dimensional Hilbert space $M$ and so by Gleason's theorem there is an
orthonormal basis $\{e_1,e_2,e_3\}$ of $M$ and real numbers $\lambda_1\ge\lambda_2\ge\lambda_3$ such that
$s'([x])=\sum_{i=1}^{3}\lambda_i|\langle x,e_i\rangle|^2$
for every unit vector $x\in M$.
Since $s'$ is not constant on the atoms, it follows that $\lambda_1>\lambda_3$.  Let $U$ be a unitary operator on $S$ such
that $Uu=e_1$,  $Uv=e_3$ and $Uw=e_2$.
Then
\[s^U([x])= s(M)s'([Ux])=s(M)\sum_{i=1}^{3}\lambda_i\langle Ux,e_i\rangle=s(M)\bigl(\lambda_1|\langle x,u\rangle|^2+
\lambda_2|\langle x,w\rangle|^2+\lambda_3|\langle x,v\rangle|^2\bigr),\]
for every unit vector $x\in M$.  Thus $s^U([a])<s(M)\lambda_1=s^U([u])$.
\end{rem}

\begin{rem}  Note that the assumption on $\dim S$ is needed solely to show {\rm(i)}$\Rightarrow${\rm(ii)}.  Our method 
seems to be inadequate for
 bigger dimensions due to the fact that the proof of Proposition \ref{p:1} relies on an elementary linear-algebraic argument
 that does not extend to the case when $F$ is not finite dimensional (as demonstrated in Remark \ref{r:4}).  Consequently,
 the induction in Lemma \ref{l:13} cannot be replaced by a transfinite
 one.  We state as an open problem the following question.
\begin{customprob}{25.i}
  Let $(S,\el)$ be a pre-Hilbert space logic such that $\dim S>\aleph_0$.  Suppose that $\tau_o(\el)$ is Hausdorff.  Does it follow that $S$ is
  complete?
\end{customprob}

All we offer on this regard are the following observations.

\begin{customcor}{25.ii}
Let $(S,\el)$ be a pre-Hilbert space logic such that $S$ is incomplete.
\begin{enumerate}[{\rm(i)}]
\item If $\el$ is atomic $\sigma$-complete then $\tau_o(\el)$ is not a Hausdorff topology.  In particular, $\tau_o(\f(S))$ is not Hausdorff.
\item If $\el$ contains a splitting incomplete subspace with countable orthogonal dimension then $\tau_o(\el)$ is not a Hausdorff topology.
\end{enumerate}
\end{customcor}
\begin{proof}
{\rm(i)}~Let $\mathcal U$ be a nonempty open subset of $\el$.   By Lemma \ref{l:13} the set
\[\mathscr S:=\left\{\left(\bigcup_n V_n\right)^{\oc S\oc S}:(V_n)\text{ is an increasing sequence in $\mathcal U$ and }\dim V_n<\infty\ \forall n\right\}\]
is an upward directed family of orthogonally closed linear subspaces of $S$ satisfying $(\cup\mathscr S)^{\oc S}=\{0\}$.
Since $\el$ is atomicly $\sigma$-complete $\mathscr S\subseteq \el$ and therefore $\mathscr S\subseteq \overline{\mathcal U}$.  Hence
$S\in\overline{\mathcal U}$.

{\rm(ii)}~Let $U\in\el\cap\e(S)$.  We first show that $\p(U)\subseteq \el$.
Let $F$ be a finite dimensional linear subspace of $U$. Obviously $F\in\el$. Moreover
$$F^{\perp_U}=U\cap F^{\perp_S}=U\wedge F^{\perp_S}=(U^{\perp_S}\vee F)^{\perp_S}\in\el.$$
In view of Theorem \ref{t:1} we have $\tau_o(\el)|_{\p(U)}\subseteq\tau_o(\p(U)).$
If $U$ is incomplete and satisfies $\dim S=\aleph_0$, then $\tau_o(\p(U))$ is not Hausdorff by the Main Theorem \ref{thm}
applied to $U$; consequently $\tau_o(\el)$ cannot be Hausdorff.
\end{proof}

As regards to {\rm(ii)} we point out however that -- as the following example shows -- there are pre-Hilbert spaces in which the only splitting linear
subspaces having countable orthogonal dimension are the finite dimensional ones.

\begin{customexmp}{25.iii}\emph{Construction of a pre-Hilbert space with the property that every splitting linear subspace with countable orthogonal dimension is
finite dimensional.}
The general idea follows that of \cite[Theorem 2.2.8]{PtWe}.

Let $S$ be a dense hyperplane of a Hilbert space. First note that then $\e(S)=\ci(S)$ as observed in\cite[Proposition 2.2.3]{PtWe} and therefore
either $\overline{U}=U\subseteq S$ or $\overline{U}^{\perp_H}=U^{\perp_S}\subseteq S$.
In addition, if  $A$ is a linear subspace of $S$, then $d(\overline A/A)\leq 1$ and therefore $A$ has an orthonormal basis
 by \cite[Corollary 9]{BCW1}; in particular, $\dim A\le \aleph_0$ if and only if $A$ is separable.
So the idea is to construct a dense hyperplane $S$ of a Hilbert space $H$ such that $A\nsubseteq S$ and $A^{\oc{H}}\nsubseteq S$ whenever
$A$ is an infinite dimensional, separable, closed linear subspace of $H$.


Let $H$ be a Hilbert space with orthogonal dimension equal to $2^{\aleph_0}$.  Denote by $\mathcal U$ the family of all the closed, separable, infinite dimensional
linear subspaces of $H$.  Then $|\mathcal U|=2^{\aleph_0}$.
Let $\preceq$ denote the smallest well-ordering on $\mathcal U$; i.e.
$[\leftarrow,A]:=\{B\in\mathcal U:B\preceq A\}$ has cardinality less than $2^{\aleph_0}$ for every $A\in\mathcal U$.  We use transfinite induction to construct a
linearly independent set of vectors $\bigcup_{A\in \mathcal U}\{x_A,y_A\}$ such that $x_A\in A$ and $y_A\in A^{\oc{H}}$ for every $A\in\mathcal U$.  For the
induction step, we suppose that the $\bigcup_{B\prec A}\{x_B,y_B\}$ has been constructed.  Bearing in mind that $d(A)=d(A^{\oc{H}})=2^{\aleph_0}$ and that the cardinality
 of $\bigcup_{B\prec A}\{x_B,y_B\}$ is less than $2^{\aleph_0}$,  one observes that there exist vectors $x_A\in A$ and $y_A\in A^{\oc{H}}$
 such that $\bigcup_{B\preceq A}\{x_B,y_B\}$ is linearly independent.
Using Zorn's lemma we can enlarge $\bigcup_{A\in \mathcal U}\{x_A,y_A\}$ to a Hamel basis $D$ of $H$.  Without loss of generality we can suppose that the vectors of
$D$ are unit vectors.  Let $g$ be a bijection between $D$ and $(0,\infty)$ and let $f$ denote the linear functional on $H$ induced  by the map
$D\to (0,\infty):x\mapsto g(x)$.  Then $f$ is an unbounded linear functional  and $S:=\mbox{ker}(f)$ is a dense hyperplane of $H$ satisfying
$A\nsubseteq S$ and $A^{\oc{H}}\nsubseteq S$ for every $A\in\mathcal U$.  Consequently, every splitting linear subspace of $S$ with a countable orthogonal dimension is  finite dimensional.
\end{customexmp}
\end{rem}

\begin{rem} Let  $(S,\el)$ be a pre-Hilbert space logic.  The set of all states on $(S,\el)$ -- being a closed
convex subset of the cube $[0,1]^\el$ -- is a compact convex topological space;  let us denote it by $(S,\el)^\ast$.
The description of the state-space of the various classical pre-Hilbert space logics has its foundations in the celebrated Gleason Theorem.

A state $s\in (S,\el)^\ast$ is said to be:
\begin{itemize}
  \item \emph{$\sigma$-additive}, if $s\bigl(\bigvee_{i\in\N}A_i\bigr)=\sum_{i\in\N}s(A_i)$ holds for every pairwise orthogonal sequence $(A_i)$ in
  $\el$ satisfying $\bigvee_{i\in\N}A_i\in\el$;
  \item \emph{regular}, if for every $A\in \el$ and every $\varepsilon>0$ there exists a finite-dimensional
  linear subspace $F$, contained in $A$, such that $s(F)>s(A)-\varepsilon$;
  \item \emph{free} or \emph{singular}, if $s(F)=0$ for every finite dimensional linear subspace $F$ of $S$.
\end{itemize}
If $S$ is a Hilbert space with $\dim S>2$,  by the Generalized Gleason Theorem,
every state on $(S,\el)$ can be lifted to a positive and normalized linear functional on the algebra
of bounded operators on $S$ (see, for example, \cite{Ha}).   Moreover, the state is
$\sigma$-additive or regular if and only if the corresponding linear functional is o-continuous.

The description of the state space of a pre-Hilbert space logic, for the incomplete case was
initiated in \cite{HaPt}: $S$ is complete if, and only if, $(S,\f(S))^\ast$ contains a $\sigma$-additive
state.  This  was shown to be also true for  $(S,\e(S))$ in \cite{DvPu}. In \cite{DvNePu} it was shown that  $(S,\f(S))$
admits a regular state if and only if $S$ is complete.  This is in contrast to $(S,\e(S))$: the map $A\mapsto \overline A$ defines a function from $\e(S)$ into
$\e(\overline S)$ satisfying that if $A,B\in\e(S)$ and $A\bot B$, then $\overline A\bot \overline B$ and $\overline{A\vee B}=\overline A\vee \overline B$;
thus the function $A\mapsto \Vert P_{\overline A} u\Vert^2$ defines a regular state on $\e(S)$ for every unit vector $u\in \overline S$.
All of these results make use of Gleason's Theorem and the
Amemiya-Araki Theorem: one applies Gleason's Theorem to force orthomodularity of $\f(S)$  and then deduce that
$S$ is complete by invoking the Amemiya-Araki Theorem.  A key observation is the following: \emph{if $s\in(S,\el)^\ast$ and $A,B\in\el$ satisfy $A\subseteq B$ and $A^{\oc{B}}=\{0\}$ then
\[0=s(A^{\oc{B}})=s((A\vee B^{\oc{S}})^{\oc{S}})=1-s(A\vee B^{\oc{S}})=1-s(A)-1+s(B)=s(B)-s(A),\]
i.e. $s(A)=s(B)$}.  This observation is fully exploited in the following theorem.
\begin{customthm}{26.i}
$(S,\el)^\ast$ has a non-free state if and only if $\el\subseteq\e(S)$.
\end{customthm}
\begin{proof}
$(\Rightarrow)$ Suppose that $B\in \el$ and $B\notin \e(S)$, i.e. there exists nonzero $v\in S$ such that
$(B+[v])\cap B^{\oc{S}}=(B^{\oc{S}}+[v])\cap B=\{0\}$.
So $A:=[v]^{\oc{S}}\cap B=([v]+B^{\oc{S}})^{\oc{S}}$
 and $C:=B+[v]$ satisfy that $A\subseteq B\subseteq C$, $A \bot [v]$,  $A^{\oc{B}}=B^{\oc{C}}=\{0\}$ and $A^{\oc{C}}=[v]$.
 Thus $s(A)=s(B)=s(C)$ and $s(C)=s(A+[v])=s(A)+s([v])$. Therefore $s([v])=0$ for every $s\in(S,\el)^\ast$.  But if $s_0$ is some state on $(S,\el)$
that is non-zero on some finite dimensional linear subspace, i.e. $s_0([u])\neq 0$ for some vector $u\in S$, then $s_0^U([v])=s_0([u])$ for a
suitable $U\in\uu(S)$.  \\
$(\Leftarrow)$ If $\el\subseteq \e(S)$, every vector state on $\e(S)$ restricts to a regular (and therefore non-free) state on $\el$.
\end{proof}

Combining this with the Amemiya-Araki Theorem one  readily gets the result of \cite{BuCh}.
\end{rem}

\begin{rem}
The assumption $\dim S=\aleph_0$ of Theorem \ref{thm} is weaker than separability of $S$: It is easy to show that a separable pre-Hilbert space
$S$ has countable  orthogonal dimension (and both conditions are equivalent if $S$ is complete). On the other hand, for any cardinal number $\kappa$
with $\aleph_0<\kappa\leq 2^{\aleph_0}$ there exists a pre-Hilbert space  having a countable orthogonal dimension with topological density equal to
$\kappa$ (see for example \cite[Example 12]{BCW1}); in particular such a pre-Hilbert space is not separable.

Under the stronger assumption of separability we can say the following:  It is easy to see that if
$A_\gamma\uparrow A$  in $\p(S)$ there is an increasing sequence
$(\gamma_i)$ such that
$A_{\gamma_i}\uparrow A$  in $\p(S)$.  A straightforward induction shows, using the orthomodularity of $\p(S)$,  that
\[A_{\gamma_k}=\bigvee_{i=1}^k A_{\gamma_{i-1}}^{\oc{A_{\gamma_{i}}}},\]
where we're letting $A_{\gamma_0}:=\{0\}$.  Hence, $A=\bigvee_{i\in\N}A_{\gamma_{i-1}}^{\oc{A_{\gamma_{i}}}}$.
Thus, if $\el$ is a pre-Hilbert space logic
associated with $S$ and $s$ is a $\sigma$-additive state
on $\el$, since $s$
restricts to a $\sigma$-additive state on $\p(S)$, it follows that $s|_{\p(S)}$ is o-continuous, i.e. $s|_{\p(S)}\in (S,\p(S))_\ast$.
(On the other hand, it is easy to verify that if $s\in(S,\el)_\ast$, then $s$ is $\sigma$-additive on $\el$.)
Thus, in the case when $S$ is separable
the conditions {\rm(i), (ii)} and {\rm(iii)} of the Main Theorem are equivalent to:
\begin{enumerate}[{\rm(iv)}]
  \item $(S,\el)$ admits a $\sigma$-additive state.
\end{enumerate}
 \end{rem}

\end{document}